\documentclass[12pt,electronic]{amsart}                                                 
\usepackage[T1]{fontenc}
\usepackage[applemac]{inputenc}
\usepackage[english]{babel}
\usepackage[small,nohug]{diagrams}
\usepackage{url}
\usepackage[hmargin=3.7cm,vmargin=3.5cm]{geometry} 
\usepackage{amsmath}
\usepackage{amsthm}                      
\usepackage{amssymb}
\usepackage{amscd}     
\usepackage{mathrsfs}
\usepackage{enumerate}
\usepackage{tikz}
\usetikzlibrary{matrix}

\usepackage[small,it]{caption}
\theoremstyle{plain}                                                           
\newtheorem{thm}{Theorem}[section]
\newtheorem{lem}[thm]{Lemma}
\newtheorem{prop}[thm]{Proposition}
\newtheorem{cor}[thm]{Corollary}
\theoremstyle{definition}
\newtheorem{ex}[thm]{Example}
\newtheorem{rem}[thm]{Remark}

\DeclareMathOperator{\SL}{SL}

\DeclareMathOperator{\Sym}{Sym}

\DeclareMathOperator{\Aut}{Aut}                          

\DeclareMathOperator{\sgn}{sgn}
\DeclareMathOperator{\Ind}{Ind}

\newcommand{\gr}{\mathfrak{gr}}
\newcommand{\ud}{\mathrm{d}}
\newcommand{\field}[1]{\ensuremath{\mathbf{#1}}}
\newcommand{\R}{\ensuremath{\field{R}}}        
\newcommand{\Q}{\ensuremath{\field{Q}}}        
\newcommand{\C}{\ensuremath{\field{C}}}
\newcommand{\sym}{\ensuremath{\mathbb{S}}}
\newcommand{\Z}{\ensuremath{\field{Z}}}

\newcommand{\st}{\; | \;}
\newcommand{\M}{\mathcal{M}}
\newcommand{\MM}{\overline{\mathcal{M}}}
\renewcommand{\S}{\mathbf{S}}

\newcommand{\V}{\mathbb{V}}

\newcommand{\pr}{\mathrm{pr}}
\renewcommand{\d}{\mathrm{d}}

\title{The structure of the tautological ring in genus one}
\author{Dan Petersen}                      
\thanks{The author is supported by the G\"oran Gustafsson Foundation for Research in Natural Sciences and Medicine.}
\email{danpete@kth.se}
\address{Institutionen f\"or matematik \\ KTH Royal Institute of Technology \\ 100 44 Stockholm \\Sweden}
\begin{document} 
  
 \maketitle

\begin{abstract} We prove Getzler's claims about the tautological ring of stable curves of genus one, i.e.\ that the tautological ring coincides with the even cohomology ring, and that all relations between generators follow from the WDVV relation and Getzler's relation.  \end{abstract}


\setcounter{section}{-1}

\section{Introduction}

The tautological ring $R^\bullet(\M_g)$ of the moduli space of curves was introduced by Mumford in the seminal \cite{mumfordtowards}. He proposed, as a first step towards understanding the Chow ring of the moduli space (probably an impossible task), to study some subring of the Chow ring still containing most ``geometrically natural'' classes. The tautological ring $R^\bullet(\M_g)$ was his suggested such subring. The definition of tautological ring was later extended to spaces of curves with marked points, and to (partial) compactifications, see e.g.\ the elegant definition in \cite{relativemaps}.

The tautological ring turned out to have a surprisingly rich structure, see in particular Faber's series of conjectures \cite{faberconjectures} which roughly say that $R^\bullet(\M_g)$ behaves like the even cohomology ring of a smooth projective variety of dimension $g-2$. Most important is perhaps the Gorenstein conjecture, which asserts that Poincar\'e duality holds in the tautological ring. Analogues of these conjectures have since been formulated also for marked points and (partial) compactifications \cite{pandharipandequestions}.  

As the tautological ring is defined in terms of explicit generators, understanding the tautological ring amounts to the same thing as understanding the relations between these generators. Another reason to be interested in relations between tautological classes is that every such relation translates into a universal partial differential equation for generating series of Gromov--Witten invariants; this was the original application of Getzler's relation on $\MM_{1,4}$ in \cite{ellipticgw}. The resulting differential equation was also key to Dubrovin and Zhang's proof of the genus one Virasoro conjecture in the semisimple case \cite{dubrovinzhang1,dubrovinzhang2}. A naive approach to proving also the non-semisimple case would be to find more genus one relations --- the results of this paper show that this is impossible, as there are none. 

  In genus zero, the tautological ring coincides with the cohomology ring of $\MM_{0,n}$, studied by Keel in \cite{keel}. It turns out that it is generated by boundary classes, and all relations between the generators arise from the WDVV relation (i.e.\ the fact that all three boundary points on $\MM_{0,4} \cong \mathbf P^1$ are linearly equivalent).  In \cite{ellipticgw}, Getzler  announced without proof a precise result describing the structure of the tautological ring in genus one. The later preprint \cite{nilpotent} contains computations of the operadic cohomology groups of certain modules over the Lie operad, and a claim that these computations will be used in his forthcoming proof. Nevertheless, a complete proof has not yet appeared. Moreover, several results appear in the literature that are true only modulo Getz\-ler's claims, e.g. 
\cite[Theorem 3]{graberpandharipande}, \cite[Theorem 2]{dsv}, \cite[Section 3.6]{belorousski}, \cite[Theorem 6.15]{pagani1}.
 
In this article we give a complete proof of the results announced in \cite{ellipticgw}, i.e.\ the following assertions:
\begin{enumerate}
\item The even dimensional cohomology of $\MM_{1,n}$ is generated additively by cycle classes of boundary strata. 
\item All relations among these generators are given by the WDVV equation and Getzler's relation.
\end{enumerate}

A more precise statement is the following. Consider the free $\Q$-vector space $F_0$ spanned by all stable $n$-pointed dual graphs of genus one. Then there is a surjection $F_0 \to H^{2\bullet}(\MM_{1,n})$ which maps a dual graph to the cycle class of the corresponding boundary stratum. The second statement determines $F_1 = \ker( F_0 \to H^{2\bullet}(\MM_{1,n}))$. For every genus zero vertex in a dual graph, and every choice of four incident half-edges, we get two relations between graphs of codimension one more by pulling back the WDVV relations from $\MM_{0,4}$. Similarly we get a codimension two relation for each genus one vertex with four incident half-edges using Getzler's relation. The claim is that these elements span $F_1$. 

It follows from these two results that the tautological ring is isomorphic to the even cohomology. Indeed, since the even cohomology is spanned by strata, the map $R^\bullet(\MM_{1,n}) \to H^{2\bullet}(\MM_{1,n})$ is surjective. Since the only relations between generators are WDVV and Getzler's, and the latter is known to be a rational equivalence by \cite{geometricgetzler}, we see that the map is also injective. We note in particular that the tautological ring is Gorenstein, by Poincar\'e duality. 

The main tools in the proof are Deligne's mixed Hodge theory, and the Eichler--Shimura theory which determines the cohomology of local systems on $\M_{1,1}$. The proof naturally splits into three steps, and therefore so does this paper. In the first section of this paper, we use the filtration of $\MM_{1,n}$ coming from the stratification by topological type to reduce the claims to an assertion about the cohomology of the interior $\M_{1,n}$, and its weight filtration. In the second section, we use the Leray spectral sequence for $\pi \colon \M_{1,n} \to \M_{1,1}$ and the Hodge-theoretic version of the Eichler--Shimura isomorphism to reduce this statement, in turn, to one about the cohomology of a single fiber of $\pi$, i.e.\ the configuration space $F(E,n)/E$ of $n$ points on an elliptic curve modulo translation. In the third section, we prove the required assertions about the cohomology of $F(E,n)/E$. Our proof uses the obvious isomorphism 
\[ F(E,n)/E \cong F(E \setminus \{0\},n-1),\]
and an analysis of the Cohen--Taylor spectral sequence computing the $\sym_{n-1}$-equivariant cohomology of $F(E \setminus \{0\},n-1)$. The above isomorphism implies that the resulting representations of $\sym_{n-1}$ must be restrictions of representations of $\sym_n$, which crucially allows us to deduce certain vanishing results in cohomology. 

It seems likely that it is only in the third step that our proof diverges from Getzler's unpublished one.  It appears from \cite{nilpotent} that he approaches the cohomology of $F(E,n)/E$ via a reinterpretation of the sequence of Cohen--Taylor complexes, as $n$ varies, as the differential graded $\sym$-module computing the operadic cohomology of a certain (left) module over the Lie operad. To determine the cohomology groups that we are interested in, it would suffice to approximate this module by its 2-step truncation, i.e.\ discard all terms in arity greater than $2$. This truncation is strongly related to Getzler's module $\mathsf{L}_\mathsf{H}$, and Getzler's \cite[Theorem 6.3]{nilpotent} and our Proposition \ref{p2} should play the same role in both proofs.

A remark is that we work over the complex numbers throughout for simplicity, and all cohomology groups are taken with rational coefficients unless stated otherwise. However, our methods are purely algebraic, and the results extend to positive characteristic by replacing all references to the Hodge-theoretic weight filtration with the natural weight filtration on $H^\bullet_\textrm{\'et}(\MM_{1,n}\otimes \mathbf{ \overline F}_p,\Q_\ell)$, replacing the Hodge-theoretic Eichler--Shimura isomorphism with its $\ell$-adic analogue, et cetera. 

Finally, let me also say that this paper owes an obvious intellectual debt to the work of Ezra Getzler, not only in the statement of the theorem but also its proof. Moreover, I benefitted greatly from a helpful conversation with him at the Park City Mathematics Institute in 2011.

\section{The filtration by topological type}\label{first}
In this section, we show by a spectral sequence argument that Getzler's claims are implied by statements about the weight filtration of the cohomology of the interior $\M_{1,n}$. We state these results on the weight filtration in this section without proof; their proofs will occupy the remainder of the paper.

Let $X = T_n \supset T_{n-1} \supset \cdots \supset T_{-1} = \varnothing$ be a filtered algebraic variety, with $T_p$ closed in $X$ for all $p$. Then there is a spectral sequence \cite[Lemma 3.8]{arapura}
\[ E_1^{pq} = H^{p+q}_c(T_p \setminus T_{p-1}) \implies H^{p+q}_c(X)\]
whose differentials are compatible with the natural mixed Hodge structures on all terms. When $n=1$, we simply have a variety $X$ with a closed subset $T_0$, and the spectral sequence coincides with the long exact sequence of mixed Hodge structures \cite[Corollary 5.51]{peterssteenbrink}
\[ \cdots \to H^\bullet_c(X\setminus T_0) \to H^\bullet_c(X) \to H^\bullet_c(T_0) \to H^{\bullet+1}_c(X \setminus T_0) \to \cdots,\]
with the connecting homomorphism giving the differential.
The general case can be deduced from the case $n=1$, in a similar way as the spectral sequence associated to a filtered complex can be constructed by iteratively applying the long exact sequence associated to a complex with a two-term filtration. (The latter being just the snake lemma.)  

We apply this spectral sequence to the case where $X = \MM_{1,n}$ and $T_p$ is the union of all strata of dimension at most $p$ in the stratification by topological type. Observe that $T_p \setminus T_{p-1}$ is the disjoint union of all strata of dimension $p$. Each stratum is given by 
\[ \M(\Gamma) = \left(\prod_{v \in \mathrm{Vert}(\Gamma)} \M_{g(v),n(v)} \right)\big/ \Aut(\Gamma),\]
where $\Gamma$ is the dual graph of the stratum. 

Now since $\MM_{1,n}$ is smooth and proper, we have that $H^{p+q}_c(\MM_{1,n}) = H^{p+q}(\MM_{1,n})$ is pure of weight $p+q$. It follows then that only $\gr^W_{p+q} E_1^{pq}$ can survive to the $E_\infty$ page of the spectral sequence. Moreover, since $H^k_c(-)$ always has weight \emph{at most} $k$, we see that to determine $\gr^W_{p+q}H^{p+q}_c(\M_{1,n})$, we need only to determine $\gr^W_i H^i_c \M_{0,n}$ and $\gr^W_{i} H^i_c \M_{1,n}$ for all $i$. 
By Poincar\'e duality, this is the same as determining $\gr^W_i H^i \M_{0,n}$ and $\gr^W_{i} H^i \M_{1,n}$.

In genus zero, the weights of the cohomologies behave in a simple way. The cohomology group $H^i\M_{0,n}$ is pure of type $(i,i)$ (and so has weight $2i$). This is a general fact about complements of arrangements of hyperplanes in $\C^N$ \cite{groupesdetresses,hyperplanepurity}; note that 
\[ \M_{0,n} \cong \{ (x_1,\ldots,x_{n-3}) \st x_i \neq x_j \text{ for any } i \neq j; x_i \neq 0, x_i \neq 1 \text{ for all } i\} \subset \C^{n-3}.\]
In particular we have that
\[\gr^W_i H^i \M_{0,n} \neq 0 \qquad \iff \qquad i=0.\]
In genus one, the corresponding statement is false, since e.g.\ the cusp form classes in $H^{11}(\M_{1,11})$ have weight $11$. However, we shall prove the following result:
\begin{thm} \label{thma} $\gr^W_{2i}H^{2i}\M_{1,n}\neq 0$ if and only if $i=0$. \end{thm}

In other words, the only \emph{even} pure cohomology classes in genus zero and one are those that are pure for trivial reasons, i.e.\ the fundamental classes. It follows that the same holds for each stratum, i.e.\ $\gr^W_{2i} H^{2i} \M(\Gamma) \neq 0 \iff i=0$. 
So when $p+q$ is even, we have that $E_1^{pq}$ is pure of weight $p+q$ along the diagonal $p=q$ and has lower weight otherwise, so that $H^{2p}_c(\MM_{1,n}) \cong E_\infty^{p,p}$. Since there are no nonzero classes on the $E_1$ page above the diagonal $p=q$, the spectral sequence furnishes us with boundary maps 
\[ E_{1}^{p,p} = H^{2p}_c(T_p \setminus T_{p-1}) \hookleftarrow E_\infty^{p,p} \stackrel \sim \leftarrow H^{2p}_c(\MM_{1,n}).\]
Since both spaces are smooth, this dualizes to a Gysin map which is therefore surjective, 
\[ H^0(T_p \setminus T_{p-1}) \to H^{2n-2p}(\MM_{1,n}). \]
The image of the fundamental class of a stratum under this map is the usual cycle class of (the closure of) the stratum. We deduce the following corollary:

\begin{cor}The even cohomology, $H^{2\bullet}(\MM_{1,n})$, is spanned by cycle classes of strata.\end{cor}

We want to understand also the space of relations between these generators, so we need to keep track of what happens to $E^{p,p}_1$ throughout the spectral sequence. We claim now the following:
\begin{enumerate}[(I)]
\item the only nontrivial differentials emanating from the diagonal $E^{p,p}_r$ occur at the $E_1$ and $E_2$ pages;
\item the strata classes that are killed on the $E_1$ page are related by an application of the WDVV equation;
\item the strata classes that are killed on the $E_2$ page are related by an application of Getzler's relation.
\end{enumerate} 
In fact, by thinking about how the spectral sequence is defined, one sees that a codimension $r$ relation should correspond to a differential on the $E_r$ page, so the above results are completely expected. 

Note that the differential on the $r$th page starting at $E^{p,p}_r$ lands in 
\[ \gr^W_{2p} E_1^{p+r,p-r-1} = \gr^W_{2p} H^{2p-1}_c(T_{p+r} \setminus T_{p+r-1}), \]
which is Poincar\'e dual to $\gr^{W}_{2r}H^{2r-1}(T_{p+r} \setminus T_{p+r-1})$. We thus wish to show that the latter space is nonzero only when $r \in \{1,2\}$. 

The result on the weights of the cohomology of $\M_{0,n}$ quoted above shows that 
$\gr^W_{i+1}H^i\M_{0,n}$ is nonzero if and only if $i=1$ (and $n\geq 4$). We will prove the following theorem:

\begin{thm} \label{thmb}$\gr^{W}_{2i} H^{2i-1} \M_{1,n}\neq 0$ if and only if $i = 2, n \geq 4$. \end{thm}

Together, these two bounds on the weight imply claim (I) above. Moreover, statement (II) now follows from the corresponding claim in genus zero. Indeed, the statement that all relations between boundary cycle classes in $H^\bullet \MM_{0,n}$ are given by the WDVV relation can be reformulated as asserting that $H^1\M_{0,n}$ is spanned additively by classes pulled back from $H^1 \M_{0,4}$ \cite[Theorem 7.3]{kontsevichmanin}\cite{getzler94}, and we have seen that all codimension one relations also in $g=1$ arise from $H^1\M_{0,n}$. The corresponding statement that needs to be shown in genus one is the following:

\begin{thm} \label{thmc} $\gr^W_4 H^3 \M_{1,n} $ is spanned additively by classes pulled back from $\gr^W_4 H^3 \M_{1,4}$ (which is one-dimensional, spanned by Getzler's relation). \end{thm}

Together, Theorems \ref{thma}, \ref{thmb} and \ref{thmc} imply all of Getzler's claims in genus one.

\section{Cohomology of local systems}
The goal of this section is to reduce the computation to understanding the cohomology of the configuration space of points on a single elliptic curve, i.e.\ a single fiber of $\M_{1,n+1}\to\M_{1,1}$. Our calculations, although independent of his, overlap somewhat with those of Gorinov \cite{gorinov1}.

\renewcommand{\R}{\mathrm{R}}

Consider a map $f \colon X \to Y$. There is a Leray spectral sequence
\[ E_2^{pq} = H^p(Y,\R^q f_\ast \Q) \implies H^{p+q}(X),\]
which in particular defines edge maps $H^q(X) \to H^0(Y,\R^qf_\ast \Q)$. If $Y$ is connected, $y \in Y$ is a point, and we assume that $\R^qf_\ast \Q$ is a local system, then $H^0(Y,\R^qf_\ast \Q ) = H^q(X_y)^{\pi_1(Y,y)}$.  If this sequence degenerates at $E_2$, then the edge maps are surjective by general considerations; when $f$ is smooth and proper this is the invariant subspace theorem. 

Let $2g-2+n>0$ and consider the forgetful map $\pi \colon \M_{g,n+k} \to \M_{g,n}$. The fiber over a geometric point of $\M_{g,n}$, corresponding to the Riemann surface $\Sigma$ with marked points $[n] \hookrightarrow \Sigma$, is the configuration space $F(U, k)$ of $k$ points in the complement $U = \Sigma \setminus [n]$. The map $\pi$ is topologically a locally trivial fibration: it suffices to prove this for $k=1$, in which case it follows by embedding $\M_{g,n+1}$ in the universal curve over $\M_{g,n}$, which is a locally trivial fibration by Ehresmann's theorem. In any case, this implies that the sheaves $\R^q \pi_\ast \Q$ are local systems whose stalks over the point $[n] \hookrightarrow \Sigma$ are $H^q(F(U,k),\Q)$. One can also define these local systems as follows: the diagonal action of $\mathrm{Diff}^+(\Sigma,[n])$ on $\Sigma^k$ preserves the subspace $F(U,k)$, thus inducing an action of $\mathrm{Diff}^+(\Sigma,[n])$ on $H^\bullet F(U,k)$, which factors through the mapping class group $\mathrm{MCG}(\Sigma,[n]) = \pi_1(\M_{g,n})$ since isotopic diffeomorphisms induce the same map on cohomology. However, the former definition makes it clear that $\R^q\pi_\ast\Q$ carries an admissible variation of mixed Hodge structure. 

We shall only concern ourselves with the fibration $\pi \colon \M_{1,n+1} \to \M_{1,1}$. Since $\M_{1,1}$ is a non-compact curve, the Leray spectral sequence for $\pi$ has only two nonzero columns and must immediately degenerate. It therefore reduces to the exact sequence
\begin{equation*} \label{exactseq}
0 \to H^1(\M_{1,1},\R^{q-1}\pi_\ast\Q) \to H^q\M_{1,n+1} \to H^q (F(U,n))^{\SL_2(\Z)} \to 0, \tag{$\ast$}
\end{equation*}
where $U$ is the complement of the origin in an elliptic curve. This is an exact sequence of mixed Hodge structures, as was first proven by Zucker \cite[9.16(ii)]{zucker}. 

Let $\V_k$ denote both the $\SL_2(\Z)$-representation $\Sym^k H^1(U)$, and the corresponding variation of Hodge structure on $\M_{1,1}$. The cohomology group $H^1(\M_{1,1},\V_k)$ is determined by the Eichler--Shimura isomorphism, or more specifically its Hodge-theoretic interpretation.  With no level structure present, one gets the following result. If $k$ is odd, then $\V_k$ has no cohomology. (This is a consequence of the fact that the elliptic involution acts as $(-1)^k$ on every fiber of $\V_k$.) If $k$ is even, then there is a decomposition 
\[ H^1(\M_{1,1},\V_k)\otimes \C = H^{k+1,0}\oplus H^{0,k+1}\oplus H^{k+1,k+1}, \]
where the first two summands are cohomology classes associated to holomorphic and antiholomorphic cusp forms of weight $k+2$ for the full modular group, and the final summand is associated to Eisenstein series of weight $k+2$. These calculations are classical and can be found in \cite{shimurabook,verdiershimura}, although when these references were written, the mixed Hodge theory needed to state the result this way did not yet exist. A detailed proof in terms of mixed Hodge theory can be found in \cite{gorinov1}.

All we shall require from the Eichler--Shimura theory is the following obvious corollary:

\begin{lem}\label{weightlem}If $\gr^W_i  H^1(\M_{1,1},\V_k) \neq 0$, then either $i$ is odd or $i \geq k+4$. \end{lem}

\begin{proof}Let $\alpha \in H^1(\M_{1,1},\V_k) \otimes \C$. If $\alpha$ is a class associated to cusp forms, then its weight is odd, and if $\alpha$ is a class associated to Eisenstein series, its weight is $2k+2$, hence at least $k+4$, since there are no Eisenstein series for $\SL_2(\Z)$ of weight less than $4$.\end{proof}

Consider $H^qF(U,n)$ as an $\SL_2(\Z)$-representation and mixed Hodge structure. There exists a direct sum decomposition
\[ \tag{$\ast\ast$}\gr^W_iH^qF(U,n) = \bigoplus_j \V_{k_j}(-n_j) \label{dec}\]
with $k_j + 2n_j = i$ for all $j$. The existence of such a decomposition can be seen very explicitly from the Cohen--Taylor spectral sequence, cf.\ the next section. Since $F(U,n)$ is smooth, $\gr^W_iH^qF(U,n)$ is nonzero only if $i\geq q$. 

 \begin{rem}Gorinov's preprint \cite{gorinov1} proves that $\R^q \pi_\ast \Q$ is in fact equal to its associated graded for the weight filtration, i.e.\ it is a direct sum of variations of (pure) Hodge structures. We shall not need this. \end{rem}

\begin{prop}\label{red}Suppose that $i$ is even and that $i \leq q+2$. Then the exact sequence \eqref{exactseq} induces an isomorphism
\[ \gr^W_i H^q \M_{1,n+1} \stackrel \sim \to \gr^W_i H^q (F(U,n))^{\SL_2(\Z)}.\]
 \end{prop}

\begin{proof}Apply the exact functor $\gr^W_i$ to \eqref{exactseq}. It suffices to show that 
\[ \gr^W_i H^1(\M_{1,1},\gr^W_j \R^{q-1}\pi_\ast\Q) = 0\]
for all $j$. Since $i$ is assumed even, Lemma \ref{weightlem} and the decomposition \eqref{dec} imply that any nonzero class in the latter space has weight at least $j+4$. On the other hand, we have already noted that smoothness of $F(U,n)$ implies that $\gr^W_j \R^{q-1}\pi_\ast\Q$ is nonzero only for $j \geq q-1$. \end{proof}

Proposition \ref{red} reduces Theorems \ref{thma}, \ref{thmb} and \ref{thmc} to statements about the weights of the cohomology of $F(U,n)$. 

\begin{rem}We remark that Proposition \ref{red} proves in particular that all classes in $\gr^W_i H^q \M_{1,n+1}$ are of Tate type when $i \leq q+2$ and $i$ is even, a fact which may be of independent interest. \end{rem}

\section{The configuration space of points on an elliptic curve}

We now wish to study the cohomology of $F(U,n)$, considered as a Hodge structure and $\SL_2(\Z)$-module, with $U$ the complement of the origin in an elliptic curve.

We compute the cohomology of $F(U,n)$ by means of the Cohen--Taylor spectral sequence \cite{cohentaylor}. When $X$ is an oriented manifold,  this is just the Leray spectral sequence for the inclusion $F(X,n) \subset X^n$, as shown by Totaro \cite{totaro}. We remark that a `Poincar\'e dual' version of this spectral sequence exists in much greater generality, as proven in \cite{benderskygitler} (for arbitrary simplicial sets) and in \cite{getzler99} (for the relative configuration space of a morphism of locally compact spaces, with arbitrary coefficient systems). To get this spectral sequence from Getzler's result, take the hypercohomology spectral sequence of the resolution he constructs of $j_!j^\ast \Z$.

\subsection{The Cohen--Taylor spectral sequence}

Let $X$ be an oriented manifold of dimension $r$. For $1\leq i,j \leq n$, let $\pr_{i} \colon X^n \to X$ be the projection on the $i$th factor, and $\pr_{ij} \colon X^n \to X^2$ similarly. We denote by $\Delta$ the class of the diagonal in $H^r(X^2)$. Define a differential bigraded commutative algebra $E_X(n)$ by
\[ E_X(n) = \left( H^\bullet(X^n)[\omega_{ij}]_{1\leq i, j \leq n}\right) \big/ \text{relations} \]
where the relations are: \begin{itemize}
\item $\omega_{ii} = 0$;
\item $\omega_{ij}^2 = 0$;
\item $\omega_{ij} = (-1)^r\omega_{ji}$;
\item $\omega_{ij}\omega_{ik} + \omega_{jk}\omega_{ji} + \omega_{ki}\omega_{kj} = 0$;
\item $\omega_{ij} (\pr_i^\ast \alpha - \pr_j^\ast \alpha) = 0$.
\end{itemize}
$E_X(n)$ is bigraded by giving $H^i(X^n)$ degree $(i,0)$ and $\omega_{ij}$ degree $(0,r-1)$. Its differential is defined by $\d \omega_{ij} = \pr_{ij}^\ast \Delta$, and $\d \alpha = 0$ for $\alpha \in H^\bullet(X^n)$. 

\begin{prop}[Totaro] $E_X(n)$ is isomorphic as a differential bigraded algebra to the $E_r$ page of the Leray spectral sequence for $F(X,n) \hookrightarrow X^n$. This is the first nontrivial differential of the spectral sequence. \end{prop}

By the functoriality of the Leray spectral sequence we see that if $G$ acts on the space $X$, then the spectral sequence is $G \wr \sym_n$-equivariant. When we consider an elliptic curve $X$, this implies that the spectral sequence is compatible with the $\SL_2(\Z)$-action on cohomology.

\begin{ex}When $n=2$ this spectral sequence is equivalent to the long exact sequence 
\[ \cdots \to H^{\bullet+r} (X) \stackrel{\Delta_!}{\longrightarrow} H^\bullet (X^2) \longrightarrow H^\bullet (F(X,2))  \stackrel{\delta} \longrightarrow H^{\bullet+r+1} (X) \to \cdots, \]
where $\Delta_!$ is the Gysin map.  \end{ex}

It is known that the differentials in the Leray spectral sequence for a morphism of algebraic varieties are compatible with the mixed Hodge structures, either by Saito's theory (using the non-perverse $t$-structure on the derived category of mixed Hodge modules \cite[Remark 4.6.(2)]{saitomixedhodge}) or by \cite{arapura}. When the Cohen--Taylor sequence is written in the form we are using here, one needs to include a shift for the weight filtration (cf.\ the preceding example). Let $X$ be an algebraic variety of complex dimension $r$. If we tensor $E^{p,q}_X(n)$, where $q= k(2r-1)$, with $\Q(-rk)$, then the differentials are strictly compatible with the natural mixed Hodge structures involved. From this one deduces:

\begin{prop}[Totaro] If $X$ is a smooth algebraic variety with $H^i(X)$ pure of weight $i$ for all $i$, then every differential except $d_r$ vanishes, and $\gr^W_{i+k}H^i F(X,n)$ is exactly the image of $E^{i-k(2r-1),k(2r-1)}_X(n)$ in $H^\bullet E_X(n)$.  \end{prop}

In particular, when $X$ is a smooth algebraic curve with at most one puncture, we find that there is a single $E_2$-differential, and that $\gr^W_i H^i F(X,n)$ is the image of classes in the bottom row, and $\gr^W_{i+1}H^iF(X,n)$ is the image of classes in the second row.

\begin{ex}\label{ex}Figures 1 and 2 display the $E_2$ and $E_\infty$ pages computing the cohomology of $F(U,3)$, where $U$ is the complement of the origin in an elliptic curve. As before we denote $\V_k = \Sym^k H^1(U)$. These are not too hard to compute by hand. \end{ex}

\begin{figure}[h] \begin{tikzpicture}
  \matrix (m) [matrix of math nodes,
    nodes in empty cells,nodes={minimum width=5ex,
    minimum height=5ex,outer sep=-5pt},
    column sep=1ex,row sep=1ex]{
                &      &     &     & & \\
           2    &     2\Q(-2) & 2\V(-2) & \cdot & \cdot & \\
          1     &  3\Q(-1) &  6 \V(-1)  & 3\V_2(-1)+3\Q(-2) & \cdot & \\
          0     &  \Q  & 3\V &  3\V_2 + 3\Q(-1)  & \V_3 + 2 \V(-1) & \\
    \quad\strut &   0  &  1  &  2  & 3 & \strut \\};
    \draw[thick] (m-5-1.north) -- (m-5-6.north) ;
\draw[thick] (m-1-1.east) -- (m-5-1.east) ;
\end{tikzpicture}
 \caption{The $E_2$ page of the Cohen--Taylor spectral sequence for $F(U,3)$.}
\end{figure}

\begin{figure}[h] \begin{tikzpicture}
  \matrix (m) [matrix of math nodes,
    nodes in empty cells,nodes={minimum width=5ex,
    minimum height=5ex,outer sep=-5pt},
    column sep=1ex,row sep=1ex]{
                &      &     &     & & \\
           2    &     \cdot & 2\V(-2) & \cdot & \cdot & \\
          1     &   \cdot &  4\V(-1)  &  {3\V_2(-1)}+ {\Q(-2)} & \cdot & \\
          0     &  {\Q}  & 3\V &  3\V_2   & \V_3  & \\
    \quad\strut &   0  &  1  &  2  & 3 & \strut \\};
    \draw[thick] (m-5-1.north) -- (m-5-6.north) ;
\draw[thick] (m-1-1.east) -- (m-5-1.east) ;
\end{tikzpicture}
\caption{The $E_3 = E_\infty$ page of the same spectral sequence. The entries ``$\Q$'' and ``$\Q(-2)$'' on the first and second row correspond to a generator and a relation, respectively. The relation is of course Getzler's relation on $\MM_{1,4}$. }
\end{figure}
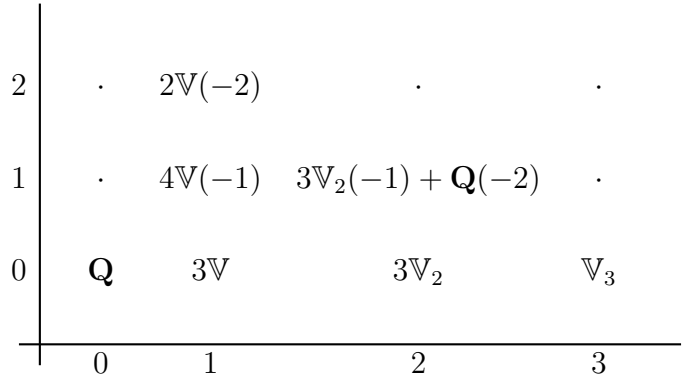

The natural forgetful maps $F(X,n+k) \to F(X,n)$ induce maps $E_X(n) \to E_X(n+k)$. These are defined by the inclusion of $H^\bullet(X)^{\otimes n}$ in $H^\bullet(X)^{\otimes (n+k)}$ and by sending $\omega_{ij}$ to $\omega_{ij}$. Clearly, $E_X(n)$ is a dg subalgebra of $E_X(n+k)$, inducing an injection in cohomology.

\begin{prop}\label{hejhej}Any class in $E^{p,k(r-1)}_X(n)$, where $n\geq p+2k$, is a sum of classes pulled back from $E^{p,q}(p+2k)$ under forgetful maps.\end{prop}

\begin{proof}If $\alpha \in E^{p,k(r-1)}$, then $\alpha$ is a sum of monomials of the form $\beta\cdot \omega_{i_1j_1}\cdots \omega_{i_kj_k}$, with $\beta \in H^p(X^n)$ homogeneous with respect to the K\"unneth decomposition. Then $\beta$ is pulled back from $H^p(X^p)$, and there are at most $2k$ distinct indices $i_l, j_l$. \end{proof}

\begin{rem}The preceding proposition is the ``surjectivity'' part of Church's proof of representation stability of $H^\bullet F(X,n)$, see \cite[Definition 2.1 II]{church}. Our proof of Theorem \ref{p2} below also uses a simple special case of his description of $E_2^{p,q}$ as an induced representation. \end{rem}

\subsection{The case of an elliptic curve}

We now specialize to the case where $U$ is the complement of the origin in an elliptic curve. 
We first want to prove the following result, which by our previous discussion implies the `generators' part of Getzler's claims. 

\begin{prop}\label{p1}For every $n$, $H^{p,0}E_U(n)^{\SL_2(\Z)} \neq 0 \iff p=0$. \end{prop}

\begin{proof}There is nothing to prove unless $p \leq n$. By Proposition \ref{hejhej}, every class in $E^{p,0}_U(n)$ is pulled back from $E^{p,0}_U(p)$, so it suffices to show that $H^{p,0}E_U(p)$ contains no $\SL_2(\Z)$-invariant classes for $p > 0$. Schur--Weyl duality gives the decomposition
\[ E^{p,0}_U(p) = H^1(U)^{\otimes p} \cong \bigoplus_{\lambda \vdash p }\S^{\lambda}( H^1(U)) \otimes V_{\lambda^T}.\]
Here $\S^\lambda$ and $V_\lambda$ denote the Schur functor and $\sym_p$-representation  corresponding to the partition $\lambda$, respectively, and $\lambda^T$ is the conjugate of the partition $\lambda$. (We get the conjugate partition because $H^1(U)$ is in odd degree.) Take $\alpha \in E^{p,0}(p)$, and suppose that there exists a transposition $\tau = (ij) \in \sym_p$ such that $\tau(\alpha) = \alpha$. Since the $\sym_2$-invariant subspace of $H^2(U^2)$ is spanned by the class of the diagonal, this implies that we can write $\alpha = \pr_{ij}^\ast (\Delta) \cdot \beta$, and then $\alpha = \ud (\omega_{ij}\cdot \beta)$. The only representation of $\sym_p$ which does not contain a copy of the trivial representation when restricted to $\sym_2$ is the alternating one, so all summands except $\Sym^{p}H^1(U) \otimes \sgn_p$ vanish in cohomology. \end{proof}

\begin{prop} \label{p2}Similarly, $H^{p,1}E_U(n)^{\SL_2(\Z)} \neq 0 \iff p=2, n \geq 3$.\end{prop}

\begin{proof}When $n = 3$ this is Example \ref{ex}, so $n \geq 4$ with no loss of generality. Again by Proposition \ref{hejhej}, it suffices to consider $H^{p,1}E_U(p+2)$. We have the decomposition
\begin{align*} 
E^{p,1}_U(p+2) \cong & \Ind_{\sym_2 \times \sym_{p}}^{\sym_{p+2}} H^0(U) \boxtimes H^p(U^p) \\
  \oplus & \Ind_{\sym_2 \times \sym_{p}}^{\sym_{p+2}} H^1(U) \boxtimes H^{p-1}(U^p).
\end{align*}

The second summand is spanned by classes of the form $\alpha \cdot \beta$, where 
\[ \alpha = \omega_{ij} \cdot \pr_i^\ast(\alpha'),\] 
with $\alpha' \in H^1(U)$, and $\beta$ is a class pulled back from $E^{p-1,0}(p-1)$.  Clearly $\ud \alpha = 0$ (as $H^3(U^2)=0$), so if $\beta$ is a coboundary then $\alpha\beta$ is, too. Hence by the preceding proposition, we get a nonzero contribution in cohomology only if $\beta$ comes from the submodule  $\Sym^{p-1}H^1(U) \subset E_U^{p-1,0}(p-1)$. Now the tensor product $H^1(U) \otimes \Sym^{p-1}H^1(U)$ can only contain an $\SL_2(\Z)$-invariant class if $p=2$ (recall from \cite[Lecture 11]{fh91} how representations of $\SL_2$ are multiplied), which gives what we wanted. 

We consider now the first summand. Let $p=2k$ (if $p$ is odd there are clearly no invariant classes). According to the Schur--Weyl decomposition, the $\SL_2(\Z)$-invariants of $H^p(U^p)$ can be written as $\Q(-k) \otimes V_{2^k}$. Hence by Pieri's formula, the $\SL_2(\Z)$-invariants of the first summand can be written as 
\[ \Q(-k-1) \otimes (V_{2^{k+1}} \oplus V_{3,2^{k-1},1} \oplus V_{4,2^{k-1}}).\]
(If $k=0$ the second two terms are zero.) Now the differential maps $E^{p,1}_U(p+2)$ to $E^{p+2,0}_U(p+2)$, and the latter contains a copy of $\Q(-k-1) \otimes V_{2^{k+1}}$ which must vanish in cohomology by our preceding result. Hence only the terms $ V_{3,2^{k-1},1} \oplus V_{4,2^{k-1}}$ could possibly be nonzero in cohomology. 

But note also now that the $\sym_{p+2}$-module $H^{p,1}E_U(p+2)^{\SL_2(\Z)}$ is in fact the restriction of a $\sym_{p+3}$-module, since we have the $\SL_2(\Z)$-equivariant isomorphism $F(U,n) \cong F(E,n+1)/E$ for any $n$. An easy combinatorial check with the branching formula for $\sym_n$ shows that neither $ V_{3,2^{k-1},1}$, $V_{4,2^{k-1}}$ nor their sum can be written as such a restriction, except for the module $V_{4}$ when $k=1$. Hence again we can get a nonzero contribution only when $p=2$. \end{proof}

This result very nearly proves also the `relations' part of the proof. However, we wanted to show that $\gr^W_4 H^3F(U,n)^{\SL_2(\Z)}$ is spanned by classes that are pulled back from $H^3F(U,3)$ along the $\binom {n+1}{4}$ forgetful maps obtained by the identification $F(U,n) \cong F(E,n+1)/E$. But since we have so far only taken into account those forgetful maps that preserve the origin, we can not get a sharp enough result: Proposition \ref{hejhej} only yields that all classes in $\gr^W_4 H^3F(U,n)$ are pulled back from $F(U,4)$. To finish the proof, we must show the following:

\begin{prop}\label{p3}The (5-dimensional) space 
$\gr^W_4 H^3 F(U,4)^{\SL_2(\Z)}$ is spanned by classes pulled back from the (1-dimensional) space $\gr^W_4 H^3 F(U,3)^{\SL_2(\Z)}$ along the five natural forgetful maps. 
\end{prop}

\begin{proof}Although a direct argument is possible, we give a quick but indirect proof. It is not hard to compute by hand with the Cohen--Taylor sequence that $\gr^W_4 H^3 F(U,4)^{\SL_2(\Z)}$ is in fact 5-dimensional. One then sees from \cite[Table 2]{getzler99} that the $\sym_5$-decomposition of $H^3(\M_{1,5}) \cong \gr^W_4 H^3 F(U,4)^{\SL_2(\Z)}$ is $V_5 \oplus V_{41}$. Since we can find at least four linearly independent pullbacks, it follows that the only way the proposition could fail is if the $\sym_5$-symmetrization of the pullback of Getzler's relation to $\MM_{1,5}$ vanishes identically (i.e.\ it is the zero relation). Direct inspection of the explicit expression for Getzler's relation in \cite{ellipticgw} shows that this is not the case. \end{proof}

Theorems \ref{thma}, \ref{thmb} and \ref{thmc} follow now from Propositions \ref{red}, \ref{p1}, \ref{p2} and \ref{p3}.

\bibliographystyle{alpha}

\bibliography{../database}

\end{document}